\documentclass{amsart}

\newtheorem{theorem}{Theorem}[section]
\newtheorem{remark}{Remark}[section]
\newtheorem{corollary}{Corollary}[section]

\title[Best approximation-preserving operators \ldots]{Best approximation-preserving operators over Hardy space}

\author{F.G. Abdullayev}
\address{Mersin University, Mersin, Turkey}
\email{fabdul@mersin.edu.tr}

\author{V.V. Savchuk}
\address{Institute of Mathematics of NAS of Ukraine, Kyiv, Ukraine,}
\email{savchuk@imath.kiev.ua}

\author{M.V. Savchuk}
\address{NTU of Ukraine “Igor Sikorsky Kyiv Polytechnical Institute”, Kyiv, Ukraine}
\email{ma.savchuk@kpi.ua}



\begin{document}

\begin{abstract}
Let $T_n$ be the linear Hadamard convolution operator acting over Hardy space $H^q$, $1\le q\le\infty$. We call $T_n$ a best approximation-preserving operator (BAP operator) if $T_n(e_n)=e_n$, where $e_n(z):=z^n,$ and if $\|T_n(f)\|_q\le E_n(f)_q$ for all $f\in H^q$, where $E_n(f)_q$ is the best approximation by algebraic polynomials of degree a most $n-1$ in $H^q$ space.

We give necessary and sufficient conditions for $T_n$ to be a BAP operator over $H^\infty$. We apply this result to establish an exact lower bound for the best approximation of bounded holomorphic functions. In particular, we show that the Landau-type inequality $\left|\widehat f_n\right|+c\left|\widehat f_N\right|\le E_n(f)_\infty$, where $c>0$ and $n<N$,  holds for every $f\in H^\infty$ iff $c\le\frac{1}{2}$ and $N\ge 2n+1$.
\end{abstract}

\maketitle

\section{Introduction} \label{s1}

Let $\mathbb D:=\{z\in\mathbb C : |z|<1\}$, $\mathbb T:=\{z\in\mathbb C : |z|=1\}$ and let $dm$ be a normalized Lebesgue measure on $\mathbb T$. The Hardy space $H^q$ for $1\le q\le\infty$ is the class of holomorphic in the $\mathbb D$ functions $f$ satisfied $\|f\|_q<\infty$, where
\[
\|f\|_q:=
\begin{cases}
\displaystyle\sup_{\rho\in(0,1)}\left(\int_\mathbb T|f(\rho t)|^qdm(t)\right)^{1/q}\hfill&\mbox{if}~1\le q<\infty,\cr
\displaystyle\sup_{z\in\mathbb D}|f(z)|\hfill&\mbox{if}~q=\infty.
\end{cases}
\]

It is well known, that for each function $f\in H^1$, the nontangential limit $f(t)$, $t\in\mathbb T,$ exist almost everywhere on $\mathbb T$ and $t\mapsto f(t)\in L^1(\mathbb T)$.

The best polynomial approximation of $f\in H^q$ is the quantity
\[ 
E_n(f)_q:=
\begin{cases}
\|f\|_q\hfill&\mbox{if}~n=0,\cr
\inf_{P_{n-1}\in\mathcal P_{n-1}}\|f-P_{n-1}\|_q\hfill&\mbox{if}~ n\in\mathbb N,
\end{cases}
\]
where $\mathcal P_{n-1}$ is the set of all algebraic polynomials of degree at most $n-1$.

Let $\{T_n\}_{n=0}^\infty$ be the sequence of bounded linear operators acting form $H^q$ into $H^q$. We call $T_n$ a \textit{best approximation-preserving operator} (BAP operator) if $T_n(e_n)=e_n$, where $e_n(z):=z^n,$ and if $\|T_n(f)\|_q\le E_n(f)_q$ for all $f\in H^q$. In case $n=0$ the operator $T_0$ is called a bound-preserving over $H^q$ [1], [2].

Clearly, if $T_n$ is a BAP operator and if $n\ge 1$, $T_n(e_k)=0$ for $k=0,1,\ldots,n-1.$ In addition, $E_n(f)_q\le\|f\|_q,$ $\forall f\in H^q.$ Thus, $T_n$ annihilates the set $\mathcal P_{n-1}$ and $\|T_n\|_{H^q\rightarrow H^q}:=\sup\{\|T_n(f)\|_q : \|f\|_q\le 1\}=1$. 

Further, we consider only the operator $T_n$ defined by Hadamard products. 

Recall that a Hadamard product of two functions $f(z)=\sum_{k=0}^{\infty}\widehat f_kz^k$ and $g(z)=\sum_{k=0}^{\infty}\widehat g_kz^k$ holomorphic in $\mathbb D$ is the function $(f*g)(z)=\sum_{k=0}^{\infty}\widehat f_k\widehat g_kz^k$, also holomorphic in $\mathbb D$. Here we denote $\widehat f_k:=f^{(k)}(0)/k!$. 
The Hadamard product has the integral representation
\[
(f*g)(z)=\int_\mathbb T f(\rho t)g\left(\frac{z}{\rho t}\right)dm(t),
\]
where $|z|<\rho<1$. If $f\in H^1$, the last formula is valid for $\rho=1$.

So, we will consider a BAP operators $T_n$ given in the forms
\[
T_n(f)=K_n*f,\quad n\in\mathbb Z_+,
\]
where a function $K_n$ is holomorphic in $\mathbb D$ and is called a kernel associated with $T_n$. 

The main reason why BAP operators are of special interest is that for a given $f\in H^q$ the convolution norm $\|K_n*f\|_q$, for a suitable $K_n$, turns out to be a sharp lower bound for the best approximation $E_n(f)_q$. For example, it was shown in [3] and [4] that the operator $T_n=K_n*$, where
\[
K_n(z)=\sum_{j=0}^{\infty}z^{jN+n}=\frac{z^n}{1-z^{N}},~n\in\mathbb Z_+,~N\in\mathbb N,
\]
is a BAP operator over $H^\infty$ if and only if $N\ge n+1$, and, moreover, for the function $f(z)=\frac{1}{1-\rho z}$, $0<\rho<1$, there holds
\[
\|T_n(f)\|_1=E_n(f)_1=\frac{2}{\pi}\rho^n{\bf K}(\rho^{n+1}),~n\in\mathbb Z_+,
\]
where 
\[
{\bf K}(x)=\int_0^{\frac{\pi}{2}}\frac{d\theta}{\sqrt{1-x^2\sin^2\theta}}
\] 
is the complete elliptic integral of the first kind. 

In view of this the main question is: \textit{what conditions on $K_n$ are necessary and sufficient for $T_n$ to be a BAP operator?}

The problem is solved only in case $n=0$. Namely, as was shown by Goluzin [5, pp. 515, 516], 
\textit{in order for $T_0$ to be a bound-preserving operator over $H^\infty$ i.e. $\|K_0*f\|_\infty\le\|f\|_\infty$, $\forall f\in H^\infty,$ it is necessary and sufficient that $2\mathrm{Re}K_0(z)\ge 1$ for all $z\in\mathbb D$.}

In this paper, we give a solution of the problem in general case.

The paper is organized as follows: In Sec.2, we give main results, which consist of two theorems. The first one gives a criterion for $T_n=K_n*$ to be a BAP operator over $H^\infty$. This criterion also implies that $T_n$ is BAP operator over $H^q$ for all $q\ge 1$. The second one, a slight refinement of previous, gives the criterion for validity of the estimate $|T_n(f)(z)|+|(L_n*f)(z)|\le E_n(f)_\infty$, where $L_n$ is a function holomorphic in $\mathbb D$ with $L_n(z)=O(z^n)$  as $z\to 0$.

In Sec.3, we concentrate on applications of main results to lower estimates for the best approximation of holomorphic functions from $H^\infty$ in terms of its Taylor coefficients. 

\section{Main results}

\begin{theorem}\label{Th_Sav-1}
Let $n\in\mathbb Z_+$, $K_n$ be a function holomorphic in $\mathbb D$, $K_n(z)=z^n+O(z^{n+1})$ as $z\to 0$ and let $T_n=K_n*$ be an operator defined as above. Then $T_n$ is a BAP operator over $H^\infty$ if and only if  
\begin{equation}\label{positiv kern}
\begin{cases}
K_n(z)=z^n+O(z^{2n+1})~\mbox{as}~z\to 0,\\
\displaystyle\mathrm{Re}\frac{K_n(z)}{z^n}\ge\frac{1}{2}~\mbox{for all}~z\in\mathbb D.
\end{cases}
\end{equation}

Moreover, (\ref{positiv kern}) implies that $T_n$ is a BAP operator over $H^q$ space for  $q\ge 1$.
\end{theorem}

\begin{proof} As was noted above, the assertion is well-known for $n=0$.	
So, further in the proof we assume $n\ge 1$. 

Let us prove the necessity. 
First of all, we note that $|T_n(f)(z)|\le\|f\|_\infty$ for all $z\in\mathbb D$, and  that $(d/dz)^k(T_n(f)(0)=0$ for $k=0,1,\ldots,n-1.$ Therefore, by Schwarz's lemma, we have
\begin{equation}\label{Schwarz est}
|T_n(f)(z)|\le|z|^n,~\forall z\in\mathbb D,
\end{equation}
for any function $f\in H^\infty$ with $\|f\|_\infty\le 1$.
	
Now let us fix $z\in\mathbb D\setminus\{0\}$ and consider the functional 
\[
\Phi_z(f):=\frac{T_n(f)(z)}{z^n}.
\]
	
According to (\ref{Schwarz est}) we get that the norm of the functional $\Phi_z$ satisfies $\|\Phi_z\|\le1.$	
On the other side, for the function $e_n$ we have $\Phi_z(e_n)=1$. Therefore $\|\Phi_z\|=1$ for any $z\in\mathbb D\setminus\{0\}$.
	
Now, let us represent $\Phi_z$ in the integral form
\begin{equation}\label{int rep}
\Phi_z(f)=\int_\mathbb Tf(t)z^{-n}K_n(\overline tz)dm(t).
\end{equation}

It is known (see [6, p. 129]), that there exists unique (extremal) $f^*\in H^\infty$ with $\|f^*\|_\infty= 1$ and there exists unique function $g_z\in H^1_0:=\{g\in H^1 : g(0)=0\}$ such that
\begin{eqnarray*}
\|\Phi_z\|&=&\left|\int_\mathbb Tf^*(t)z^{-n}K_n(\overline tz)dm(t)\right|\\
&=&\int_\mathbb T\left|z^{-n}K_n(z\overline  t)+g_z(t)\right|dm(t)
\end{eqnarray*}
and
\begin{equation}\label{extremality}
f^*(t)\left(z^{-n}K_n(\overline tz)+g_z(t)\right)=\left|z^{-n}K_n(\overline tz)+g_z(t)\right|
\end{equation}
for a.e. $t\in\mathbb T$.

Let $K_n(z)=z^n+\sum_{k=0}^{\infty}\alpha_{k,n}z^k$ be a power series expansion for $K_n$. Since the function $f^*=e_n$ is extremal for $\Phi_z$, the equality (\ref{extremality}) implies the relation
\begin{equation}\label{positivity}
t^nz^{-n}\left(z^n\overline{t}^n+\sum_{k=n+1}^{\infty}\alpha_{k,n}z^k\overline t^{k}\right)+t^ng_z(t)\ge 0
\end{equation}
for a.e. $t\in\mathbb T$. This gives
\[
\mathrm{Im}(t^ng_z(t))=\mathrm{Im}\left(\sum_{k=n+1}^{\infty}\overline{\alpha_{k,n}z^{k-n}}t^{k-n}\right)
\]
for a.e. $t\in\mathbb T$.
	
Therefore, by Schwarz's integral formula we get
\begin{eqnarray*}
t^ng_z(t)&=&\mathrm i\int_\mathbb T\mathrm{Im}(w^ng_z(w))\frac{1+\overline wt}{1-\overline wt}dm(w)	\\
&=&\sum_{k=n+1}^{\infty}\overline{\alpha_{k,n}z^{k-n}}t^{k-n}
\end{eqnarray*}
for all $t\in\mathbb D$. Consequently,
\[
g_z(t)=\sum_{k=n+1}^{\infty}\overline{\alpha_{k,n}z^{k-n}}t^{k-2n},~t\in\mathbb D.
\]
But $g_z$ must be in $H^1_0$. Hence, it follows that $\alpha_{k,n}=0$ for $k=n+1,\ldots, 2n$, or, equivalently, the first relation in (\ref{positiv kern}). Moreover, from (\ref{positivity}) follows the second relation in (\ref{positiv kern}).

To complete the proof,  we show that (\ref{positiv kern}) implies $T_n$ is a BAP operator over $H^q$ for $1\le q\le\infty$. Using (\ref{int rep}) and the equality $\int_\mathbb Tf(t)t^kdm(t)=0$ for $k\in\mathbb N$,  we get the representation 
\begin{eqnarray}\label{int rep2}
T_n(f)(z)&=&z^n\int_\mathbb Tf(t)\overline t^n\frac{K_n(\overline tz)}{(\overline tz)^n}dm(t)\nonumber\\
&=&z^n\int_\mathbb T\left(f(t)-P(t)\right)\overline t^n\left(2\mathrm{Re}\frac{K_n(\overline z t)}{(\overline zt)^n}-1\right)dm(t),~z\in\mathbb D,
\end{eqnarray}
where $P$ ia an arbitrary polynomial from $\mathcal P_{n-1}$.
The result follows by estimating the integral by Minkowski's inequality.
\end{proof}

\begin{remark}\label{rem1}
By Herglotz's theorem (see [6, p.19]) the conditions (\ref{positiv kern}) are equivalent to that
\begin{equation}\label{C-S int}
K_n(z)=z^n\int_\mathbb T\frac{d\mu(t)}{1-\overline tz},~z\in\mathbb D,
\end{equation}

here $\mu$ is a positive Borel measure on $\mathbb T$ of total variation 1 satisfying
\begin{equation}\label{moment}
\int_\mathbb Tt^kd\mu(t)=0,~k\in\mathbb Z, 1\le|k|\le n.
\end{equation}
Consequence (it follows from (\ref{int rep2}) and (\ref{C-S int})) a BAP operator $T_n$ over $H^q$, $1\le q\le\infty$, has the representation
\[
T_n(f)(z)=\int_\mathbb Tf\left(\overline tz\right)t^nd\mu(t).
\]
\end{remark}
With respect to Theorem \ref{Th_Sav-1} and Remark \ref{rem1},  naturally arises the following question: \textit{how does the condition
\[
\inf_{z\in\mathbb D}\left(\mathrm{Re}\frac{K_n(z)}{z^n}-\frac{1}{2}\right)=\frac{1}{2}\inf_{z\in\mathbb D}\int_\mathbb T\frac{1-|z|^2}{|1-\overline tz|^2}d\mu(t)=a>0
\]
influence on sharpness of the estimate $\|T_n(f)\|_\infty\le E_n(f)_\infty$ for individual function?}

The answer to this question is the following:

\begin{theorem}\label{Th_Sav-2}
Let $n\in\mathbb Z_+$ and let $K_n$ and $L_n$ be a  holomorphic functions in $\mathbb D$, and $L_n(z)=O(z^n)$ as $z\to 0$. Then $T_n=K_n*$ is a BAP operator over $H^\infty$ and 
\begin{equation}\label{thm2 ineq}
\sup_{z\in\mathbb D}\left(\left|T_n(f)(z)\right|+\left|(L_n*f)(z)\right|\right)\le E_n(f)_\infty,~\forall f\in H^\infty,
\end{equation}
if and only if
\[
\begin{cases}
K_n(z)=z^n+O\left(z^{2n+1}\right)~\mbox{as}~z\to 0,\\
L_n(z)=O\left(z^{2n+1}\right)~\mbox{as}~z\to 0,\\
\displaystyle\mathrm{Re}\frac{K_n(z)}{z^n}-\frac{1}{2}\ge\left|\frac{L_n(z)}{z^n}\right|~\mbox{for all}~z\in\mathbb D.
\end{cases}
\]
\end{theorem}

Theorem \ref{Th_Sav-2} in case $n=0$ is due to Goluzin [5, pp. 519, 520].

\begin{proof} We observe that for given $z\in\mathbb D$,
\[
\left|T_n(f)(z)+\mathrm e^{\mathrm{i}\alpha}(L_n*f)(z)\right|\le\left|T_n(f)(z)\right|+\left|(L_n*f)(z)\right|,
\] 
for any $\alpha\in\mathbb R$. Equality holds here if and only if $\alpha=\arg T_n(z)-\arg (L_n*f)(z)$. Therefore, (\ref{thm2 ineq}) is equivalent to
\begin{equation}\label{sum ineq}
\max_{\alpha\in\mathbb R}\left\|T_n(f)+\mathrm e^{\mathrm{i}\alpha}(L_n*f)\right\|_\infty\le E_n(f)_\infty,~\forall f\in H^\infty.
\end{equation}

Now, consider the family of operators $\left\{T_{n,\alpha}\right\}_{\alpha\in\mathbb R}$, defined on $H^\infty$ by 
\begin{eqnarray*}
T_{n,\alpha}(f)&=&(K_n+\mathrm e^{\mathrm {i}\alpha}L_n)*f\\
&=&T_n(f)+\mathrm e^{\mathrm i\alpha}(L_n*f).
\end{eqnarray*}
Applying Theorem \ref{Th_Sav-1} to each $T_{n,\alpha}$, one can show that (\ref{sum ineq}) together with statement that $T_n$ is a BAP operator, is equivalent to the statements
that
\begin{eqnarray*}
\mathrm e^{\mathrm {i}\alpha}L_n(z)&=&z^n-K_n(z)+O(z^{2n+1})\\
&=&O(z^{2n+1}),
\end{eqnarray*}
as $z\to 0$, and
\[
\mathrm{Re}\frac{K_n(z)}{z^n}+\mathrm{Re}\left(\mathrm e^{\mathrm i\alpha}\frac{L_n(z)}{z^n}\right)\ge\frac{1}{2}
\]
for all $z\in\mathbb D$ and for all $\alpha\in\mathbb R$. To complete the proof, we take $\alpha=-\arg L_n(z)+n\arg z+\pi$.
\end{proof}

\section{Application}

The Cauchy inequality states that
\begin{equation}\label{Cauchy ineq}
\left|\widehat f_n\right|\le \|f\|_q,~\forall f\in H^q,
\end{equation}
where $1\le q\le\infty$.
Equality (for given $n$) here is  attained for the function $e_n$.
But for bounded holomorphic functions in $\mathbb D$ the following Landau inequality is stronger than (\ref{Cauchy ineq}) [7, p. 34]:  
\begin{equation}\label{Landau ineq}
\left|\widehat f_n\right|+\frac{1}{2}\left|\widehat f_N\right|\le \|f\|_\infty,~\forall f\in H^\infty,
\end{equation}
where $n, N\in\mathbb Z_+$, and $N\ge 2n+1$. Moreover, in [8] it was shown that the constant $\frac{1}{2}$ is sharp in the sense that
\begin{equation}\label{sharp}
\sup_{f\in H^\infty, \|f\|_\infty\le 1}\frac{\left|\widehat f_N\right|}{1-\left|\widehat f_n\right|}=2,~\forall N\ge 2n+1.
\end{equation}
Later on (Corollary \ref{main col}) we will give an alternate proof to (\ref{Landau ineq}) and (\ref{sharp}).

Applying (\ref{Cauchy ineq}) and (\ref{Landau ineq}) to the function $f-p$, where $p\in\mathcal P_{n-1}$, we obtain the following:
\begin{equation}\label{first ineq}
\left|\widehat f_n\right|\le E_n(f)_q,~\forall f\in H^q,
\end{equation}
\begin{equation}\label{second ineq}
\left|\widehat f_n\right|+\frac{1}{2}\left|\widehat f_N\right|\le E_n(f)_\infty,~\forall f\in H^\infty,
\end{equation}
where $1\le q\le\infty$, $n, N\in\mathbb Z_+$, and $N\ge 2n+1$.

The inequality (\ref{first ineq}) is sharp on whole space $H^q$ in the following sense: equality in (\ref{first ineq}) for given $n$, as was shown in [9], is attained if and only if
\[
\begin{cases}
\displaystyle f\in\mathcal P_{2n}\wedge\mathrm{Re}\sum_{k=0}^{n}\frac{\widehat f_{k+n}}{\widehat f_n}z^k\ge\frac{1}{2},~z\in\mathbb D,~&\mbox{if}~q=1,\\
f\in\mathcal P_n~&\mbox{if}~1<q\le\infty,
\end{cases}
\] 
provided $\left|\widehat f_n\right|>0$.

In this section we demonstrate the application of previous results to obtain some refinements of (\ref{second ineq}) for functions from $H^\infty\setminus\mathcal P_{2n}$.

The main tool in the section is the following:

\begin{theorem}\label{Th_Sav-3} Let $n\in\mathbb Z_+$, and let $L_n$ be a holomorphic function in $\mathbb D$ such that $L_n(z)=O(z^n)$ as $z\to 0$. Then
\begin{equation}\label{ineq Cor 2}
\left|\widehat f_n\right|+\left\|L_n*f\right\|_\infty\le E_n(f)_\infty,~\forall f\in H^\infty,
\end{equation}
if and only if $|L_n(z)|\le\frac{1}{2}|z|^{2n+1}$ for all $z\in\mathbb D$.
\end{theorem}

\begin{proof} Taking $K_n(z)=z^n$, we get $T_n(f)(z)=(K_n*f)(z)=\widehat f_nz^n$. Therefore,
\[
\sup_{z\in\mathbb D}\left(\left|T_n(f)(z)\right|+\left|(L_n*f)(z)\right|\right)=\left|\widehat f_n\right|+\|L_n*f\|_\infty,
\]
and
\[
\mathrm{Re}\frac{K_n(z)}{z^n}-\frac{1}{2}=\frac{1}{2},~z\in\mathbb D
\]
Moreover, $T_n$ is a BAP operator.  Hence by Theorem \ref{Th_Sav-2}, (\ref{ineq Cor 2}) is equivalent to 
\[
\begin{cases}
L_n(z)=O\left(z^{2n+1}\right)~\mbox{as}~z\to 0,\\
\displaystyle\left|\frac{L_n(z)}{z^n}\right|\le\frac{1}{2}~\mbox{for all}~z\in\mathbb D.
\end{cases}
\]
By Schwarz lemma this is equivalent to $|L_n(z)|\le\frac{1}{2}|z|^{2n+1}$ for all $z\in\mathbb D$. 
\end{proof}

For $f\in H^1$ we set
\[
\mathcal E_{k}(f)_1:=
\begin{cases}
\displaystyle \inf_{h\in H^1_0}\|f-\overline h\|_1~&\mbox{if}~k=0,
\\
\displaystyle\inf_{p\in\mathcal P_{k-1}, h\in H^1_0}\|f-\left(p+\overline h\right)\|_1~&\mbox{if}~k\in\mathbb N.
\end{cases}
\] 

\begin{corollary}
If $n, N\in\mathbb Z_+$, $N\ge 2n+1$, and $f\in H^\infty$, then
\begin{equation}\label{ineq 3}
\left|\widehat f_n\right|+\frac{1}{2}\mathcal E_{N}(f)_1\le E_n(f)_\infty.
\end{equation}
The number $\frac{1}{2}$ cannot be improved.
\end{corollary}

\begin{proof}
It follows from (\ref{ineq Cor 2}) that
\begin{equation}\label{ineq 4}
\left|\widehat f_n\right|+\frac{1}{2}\sup_{L_n}\|L_n*f\|_\infty\le E_n(f)_\infty,
\end{equation}
where supremum is over  all functions $L_n$ holomorphic in $\mathbb D$ such that $|L_n(z)\le|z|^{N}$, $z\in\mathbb D$. Since $f\in H^\infty$ and $L_N/e_{N-1}\in H^\infty_0$, it follows that convolution
\[
\left(L_n*f\right)(z)=z^{N-1}\int_\mathbb T\frac{L_n(\overline tz)}{(\overline tz)^{N-1}}\frac{f(t)}{t^{N-1}}dm(t),~z\in\mathbb D,
\]
is continuous on the closed disc $\overline{\mathbb D}$ (see [10, pp. 37, 38]). Therefore, 
by the basic duality relation [6, ch. IV], we get
\begin{eqnarray}
\sup_{L_n}\left\|L_n*f\right\|_\infty&=&\sup_{g\in H^\infty_0, \|g\|_\infty\le 1}\left|\int_\mathbb T\overline{g(t)}\frac{f(t)}{t^{N-1}}dm(t)\right|\nonumber\\
&=&\inf_{h\in H^1}\int_\mathbb T\left|f(t)t^{-(N-1)}-\overline{h(t)}\right|dm(t)\nonumber\\
&=&\inf_{h\in H^1}\int_\mathbb T\left|f(t)-t^{N-1}\overline{h(t)}\right|dm(t)\nonumber\\
&=&\mathcal E_N(f)_1.\label{ineq 6}
\end{eqnarray}
Here we notice that for all $h\in H^1$, $t^{N-1}\overline{h(t)}=\sum_{k=0}^{N-1}\overline{h_{N-1-k}}t^k+\overline{h_1(t)}$, where $h_1\in H^1_0$.
Substituting (\ref{ineq 6}) in (\ref{ineq 4}), we obtain (\ref{ineq 3}).

Now, suppose that there exist number $c>\frac{1}{2}$ such that
\begin{equation}\label{ineq 7}
\left|\widehat f_n\right|+c\mathcal E_{N}(f)_1\le E_n(f)_\infty.
\end{equation}
Then by the theorem about existence and uniqueness of extremal elements in the duality relation [6, p. 129], there exists a unique function $\widetilde g\in H^\infty_0$ with $\|\widetilde g\|_\infty=1$ that realize the second supremum in (\ref{ineq 6}). Hence, according to (\ref{ineq 7}) and (\ref{ineq 6}), for the function $\widetilde L_n=c\widetilde ge_{N-1}$ the inequality (\ref{ineq Cor 2}) holds. By Theorem \ref{Th_Sav-3} this is equivalent to 
\[
\left|\widetilde L_n(z)\right|=c\left|\widetilde g(z)z^{N-1}\right|\le\frac{1}{2}\left|z\right|^{2n+1}
\]
for all $z\in\mathbb D$. This implies $\|\widetilde g\|_\infty\le\frac{1}{2c}<1$, a contradiction.
\end{proof}

\begin{remark}
Let $\mathcal R_N$ be a set of all functions $f$ holomorphic in $\mathbb D$ for which $\left|\widehat f_N\right|>0$ and 
\[
\mathrm{Re}\frac{1}{\widehat f_N}\sum_{k=0}^{\infty}\widehat f_{k+N}z^k\ge\frac{1}{2}
\]
for all $z\in\mathbb D$. Clearly, $\mathcal E_N(f)_1\ge\left|\widehat f_N\right|$, and, as was shown in [11], $\mathcal E_N(f)_1=\left|\widehat f_N\right|$ if and only if $f\in\mathcal R_N$. Therefore (\ref{ineq 3}) is a strengthening of (\ref{second ineq}) on the functional class $H^\infty\setminus\mathcal R_N$.
\end{remark}

The following assertion shows that the conditions for validity of Landau's inequality (\ref{Landau ineq}) as well as (\ref{second ineq})  are final.

\begin{corollary}\label{main col}
Let $c>0$, $n, N\in\mathbb Z_+$, and  $n<N$. In order that 
\begin{equation}\label{eq1}
\left|\widehat f_n\right|+c\left|\widehat f_N\right|\le E_n(f)_\infty,~\forall f\in H^\infty,
\end{equation}
it is necessary and sufficient that $N\ge 2n+1$ and that $c\le \frac{1}{2}$.

Moreover, for $N\ge 2n+1$,
\begin{equation}\label{eq2}
\sup_{f\in H^\infty\setminus\mathcal P_{n-1}}\frac{\left|\widehat f_n\right|+\frac{1}{2}\left|\widehat f_N\right|}{E_n(f)_\infty}=1.
\end{equation}
\end{corollary}

\begin{proof}
Taking $L_n(z)=cz^{N}$, we obtain
$c\left|\widehat f_N\right|=\|L_n*f\|_\infty.$
Hence, by Theorem \ref{Th_Sav-3}, (\ref{eq1}) is equivalent to $|L_n(z)|=c|z|^N\le \frac{1}{2}|z|^{2n+1}$ for all $z\in\mathbb D$. This is only if $N-(2n+1)\ge 0$ and $c\le\frac{1}{2}$

To prove (\ref{eq2}), we consider the sequence of functions $\left\{f_\rho\right\}_{0\le\rho<1}$, where
\[
f_\rho(z)=z^n\frac{z^{N-n}-\rho}{1-z^{{N-n}}\rho}.
\]
Clearly, $1=\|f_\rho\|_\infty\ge E_n\left(f_\rho\right)_\infty$, $\widehat{\left(f_\rho\right)_n}=-\rho$ and $\widehat{\left(f_\rho\right)_N}=1-\rho^2$. Therefore we obtain
\begin{eqnarray*}
1&\ge&\sup_{f\in H^\infty}\frac{\left|\widehat f_n\right|+\frac{1}{2}\left|\widehat f_N\right|}{E_n(f)_\infty}\\
&\ge&\frac{\left|\widehat{\left( f_\rho\right)_n}\right|+\frac{1}{2}\left|\widehat{\left(f_\rho\right)_N}\right|}{E_n\left(f_\rho\right)_\infty}\\
&\ge&\rho+\frac{1}{2}\left(1-\rho^2\right).
\end{eqnarray*} 	
The result follows on letting $\rho\to 1-$.
\end{proof}

\begin{corollary}\label{col 3} Let $n\in\mathbb Z_+$ and let $\{\psi_k\}$ be sequence of non-negative numbers such that 
\begin{equation}\label{bohr con}
\sum_{k=2n+1}^{\infty}\psi_k\le\frac{1}{2}.
\end{equation} Then 
\begin{equation}\label{ineq 8}
\left|\widehat f_n\right|+\sum_{k=2n+1}^{\infty}\left|\widehat f_k\right|\psi_k\le E_n(f)_\infty,~\forall f\in H^\infty.
\end{equation}
The number $\frac{1}{2}$ in (\ref{bohr con}) cannot be increased.
\end{corollary}

\begin{proof} Fix $f\in H^\infty$ and consider the function $L_n(z)=\sum_{k=2n+1}^{\infty}\psi_k\mathrm{e}^{\mathrm i\arg\widehat f_k}z^k$. We have 
\[
\sum_{k=2n+1}^{\infty}\left|\widehat f_k\right|\psi_k=\left\|L_n*f\right\|_\infty.
\]
Since 
\begin{eqnarray*}
|L_n(z)|&\le&|z|^{2n+1}\sum_{k=2n+1}^{\infty}\psi_k\\
&\le&\frac{1}{2}|z|^{2n+1},~ z\in\mathbb D,
\end{eqnarray*}
(\ref{ineq 8}) follows by Theorem \ref{Th_Sav-3}.

Let us now prove that restriction (\ref{bohr con}) cannot be weakened. Suppose that (\ref{ineq 8}) holds with $\frac{1}{2}<\sum_{k=2n+1}^{\infty}\psi_k<+\infty$. Since the function $\rho\mapsto\sum_{k=2n+1}^{\infty}\psi_k\rho^k$ is continuous and increasing on $[0,1]$, there exists a unique number $\rho_0\in(0,1)$ such that $\sum_{k=2n+1}^{\infty}\psi_k\rho^k_0=\frac{1}{2}$.
Therefore for the holomorphic  function
\[
f(z)=z^n\frac{z^{n+1}-\rho_0}{1-z^{n+1}\rho_0}=-\rho_0z^n+\frac{1-\rho^2_0}{\rho^{2n+1}_0}\sum_{k=2n+1}^{\infty}\rho^k_0z^k
\]
we have
\begin{eqnarray*}
1&\ge&E_n(f)_\infty\\
&\ge&\left|f_n\right|+\sum_{k=2n+1}^{\infty}\left|\widehat f_k\right|\psi_k\\
&=&\rho_0+\frac{1-\rho^2_0}{2\rho^{2n+1}_0}
\end{eqnarray*}
or, equivalently,
\[
1+\rho_0\le 2\rho^{2n+1}_0.
\]
On the other side,
\[
2\rho^{2n+1}_0\le\rho^{2n+1}_0+\rho^{2n+1}_0\le 1+\rho_0.
\]
Hence, only $\rho_0=1$, a contradiction.
\end{proof}

For example, if $n=0$ and if $\psi_k=\rho^k$, where $0<\rho<1$, the Corollary \ref{col 3} coincide with the famous Bohr's theorem. Indeed, the condition (\ref{bohr con}) take a form
\[
\frac{\rho}{1-\rho}\le\frac{1}{2}\Leftrightarrow\rho\le\frac{1}{3},
\]
and (\ref{bohr con}) becomes 
\[
\sum_{k=0}^{\infty}\left|\widehat f_k\right|\rho^k\le\|f\|_\infty,~\forall f\in H^\infty.
\]

\bigskip

\textbf{Acknowledgements.} This research was supported by the Kyrgyz-Turkish Manas University (Bishkek/Kyrgyz Republic), Project No. KTMU-BAP-20I9.FBE,06

\end{document}